\documentclass[11pt,reno]{amsart} 
\usepackage{amssymb}
\usepackage{amsmath}
\usepackage{enumitem}
\usepackage{mathrsfs}
\usepackage[english]{babel}
\usepackage[all]{xy}
\usepackage{hyperref}
\usepackage{marginnote}
\usepackage{mathtools}

\usepackage{tikz-cd}
\usepackage{tikz}
\usepackage{fullpage}

\usepackage{stmaryrd}

\usepackage[margin=1.25in]{geometry}

\newtheorem{theorem}{Theorem}[section]
\newtheorem{lemma}[theorem]{Lemma}
\newtheorem{proposition}[theorem]{Proposition}
\newtheorem{corollary}[theorem]{Corollary}

\theoremstyle{definition}
\newtheorem*{ack}{Acknowledgements}

\newtheorem{remark}[theorem]{Remark}

\newtheorem{definition}[theorem]{Definition}

\numberwithin{equation}{section} \numberwithin{figure}{section}

\DeclareMathOperator{\Sym}{Sym}
\DeclareMathOperator{\Spec}{Spec}

\DeclareMathOperator{\an}{an}

\DeclareMathOperator{\Hom}{Hom}

\DeclareMathOperator\Vanish{V}

\DeclareMathOperator{\supp}{supp}

\newcommand{\Tangent}{\textrm{T}}

\DeclarePairedDelimiter\ceil{\lceil}{\rceil}

\DeclareMathOperator\id{id}

\newcommand*\ratmap{\mathbin{\tikz [baseline=-0.25ex,-latex, dashed, ->] \draw [densely dashed] (0pt,0.5ex) -- (1.3em,0.5ex);}}

\usepackage{color}

\definecolor{orange}{rgb}{1,0.5,0}

\title{Kobayashi-Ochiai's finiteness theorem for   orbifold pairs of general type}

\author{Finn Bartsch}
\address{Finn Bartsch \\
IMAPP Radboud University Nijmegen \\
PO Box 9010, 6500GL \\
Nijmegen, The Netherlands\\}
\email{f.bartsch@math.ru.nl }

\author{Ariyan Javanpeykar}
\address{Ariyan Javanpeykar \\ 
IMAPP Radboud University Nijmegen \\
PO Box 9010, 6500GL\\
 Nijmegen, The Netherlands}
\email{ariyan.javanpeykar@ru.nl }



\begin{document}

\begin{abstract}
Kobayashi--Ochiai proved that the set of dominant maps from a fixed variety to a fixed variety of general type is finite. We prove the natural extension of their finiteness theorem to Campana's orbifold pairs.    
\end{abstract}
\maketitle

\thispagestyle{empty}

\section{Introduction}
In \cite{KobaOchiai} Kobayashi and Ochiai proved a higher-dimensional generalization of the finiteness theorem of De Franchis for compact Riemann surfaces. Namely, for $X$ and $Y$ smooth projective varieties over $\mathbb{C}$ with $X$ of general type, the set of dominant rational maps $Y\ratmap X$ is finite. 
In this paper we prove a generalization of the classical finiteness theorem of Kobayashi--Ochiai for dominant rational maps in the setting of Campana's orbifold maps (Theorem \ref{thm1}). The notion of orbifold pairs (also referred to as \emph{C-pairs} \cite{Smeets}) was introduced in \cite{CampanaOr0,CampanaJussieu} and has already been shown to be fruitful in, for example, the resolution of Viehweg's hyperbolicity conjecture \cite{CPaun}. \smallbreak 
  
Let $k$ be an algebraically closed field.
A variety (over $k$) is an integral finite type separated scheme over $k$. 
A \emph{$\mathbb{Q}$-orbifold (over $k$)} $(X, \Delta)$ is a variety $X$ together with a $\mathbb{Q}$-Weil divisor $\Delta$ on $X$ such that all coefficients of $\Delta$ are in $[0,1]$. If $\Delta = \sum_i \nu_i \Delta_i$ is the decomposition of $\Delta$ into prime divisors, we say that $m(\Delta_i) := (1-\nu_i)^{-1}$ is the \emph{multiplicity} of $\Delta_i$ in $\Delta$. If all multiplicities of a $\mathbb{Q}$-orbifold are in $\mathbb{Z} \cup \{ \infty\}$, we say that $(X,\Delta)$ is a \emph{$\mathbb{Z}$-orbifold} or simply an \emph{orbifold}. \smallbreak 

A $\mathbb{Q}$-orbifold $(X, \Delta)$ is \emph{normal} if the underlying variety $X$ is normal.
Moreover, a $\mathbb{Q}$-orbifold $(X, \Delta)$ is \emph{smooth (over $k$)} if the underlying variety $X$ is smooth and the support of the orbifold divisor $\supp \Delta$ is a divisor with strict normal crossings. This means that every component of $\supp \Delta$ is smooth and that étale locally around any point of $X$, the divisor $\supp \Delta$ is given by an equation of the form $x_1\cdots x_n=0$ for some $n \leq \dim X$. \smallbreak 

Let $(X, \Delta_X)$ be a normal $\mathbb{Q}$-orbifold and $(Y, \Delta_Y)$ be a $\mathbb{Q}$-orbifold such that $Y$ is locally factorial. In this case, we define a \emph{morphism} of $\mathbb{Q}$-orbifolds $f \colon (X, \Delta_X) \to (Y, \Delta_Y)$ to be a morphism of varieties $f \colon X \to Y$ satisfying $f(X) \nsubseteq \supp \Delta_Y$ such that, for every prime divisor $E \subseteq \supp \Delta_Y$ and every prime divisor $D \subseteq \supp f^*E$, we have $t m(D) \geq m(E)$, where $t \in \mathbb{Q}$ denotes the coefficient of $D$ in $f^*E$; the local factoriality of $Y$ ensures that $E$ is a Cartier divisor, so that $f^\ast E$ is well-defined. Note that, equivalently, we can require that $t - 1 + \nu_D \geq t \nu_E$, where $\nu_D$ and $\nu_E$ are the coefficients of $D$ in $\Delta_X$ and  of $E$ in $\Delta_Y$, respectively. \smallbreak 

If $X$ is a normal variety, we identify $X$ with the orbifold $(X, 0)$. If $X$ and $Y$ are varieties such that $X$ is normal and $Y$ is locally factorial, every morphism of varieties $X \to Y$ is an orbifold morphism $(X,0) \to (Y,0)$. \smallbreak 
 
A $\mathbb{Q}$-orbifold $(X, \Delta)$ is \emph{proper} (resp. \emph{projective}) if the underlying variety $X$ is proper (resp. projective) over $k$.
A smooth proper $\mathbb{Q}$-orbifold $(X,\Delta)$ is of \emph{general type} if $K_X+\Delta$ is a big $\mathbb{Q}$-divisor, where $K_X$ denotes the canonical divisor of $X$. If $\Delta=0$, we recover the usual notion of a smooth proper variety of general type. If the multiplicities of $\Delta$ are all infinite, then $(X,\Delta)$ is of general type if and only if the smooth quasi-projective variety $X\setminus \Delta$ is of log-general type. Finally, if $X$ is a smooth proper variety of nonnegative Kodaira dimension, $D$ is a strict normal crossings divisor, and $m \geq 2$, then the orbifold $(X, (1-\frac{1}{m})D)$ is of general type if and only if $X \setminus D$ is of log-general type.

\begin{theorem}\label{thm1}
 If $(Y,\Delta)$ is a smooth proper orbifold pair of general type and $X$ is a normal variety, then the set of separably dominant orbifold morphisms $X\to (Y,\Delta)$ is finite. 
\end{theorem}  

We prove a more general result in which we consider rational maps $X \ratmap (Y,\Delta)$; see Theorem \ref{thm:final} for a precise statement. \smallbreak 

Theorem \ref{thm1} (or, actually, its more precise version Theorem \ref{thm:final}) generalizes Kobayashi--Ochiai's finiteness theorem for dominant rational maps to a projective variety of general type in characteristic zero (take $\Delta$ to be trivial and $k$ of characteristic zero) \cite{KobaOchiai}. It also implies Tsushima's extension of Kobayashi--Ochiai's theorem to varieties of log-general type (take the multiplicities of $\Delta$ to be infinity and $k$ of characteristic zero) \cite{Tsushima}. Moreover, we also obtain the  finiteness theorem of Martin-Deschamps and  Menegaux   for separably dominant morphisms to a proper variety of general type (take $\Delta$ to be trivial and $k$ of arbitrary characteristic) \cite{MartinDeschampsMenegaux}, as well as Iwanari--Moriwaki's extension of Tsushima's result to characteristic $p$ \cite{IwanariMoriwaki}. Finally, we obtain a new proof of Campana's extension of De Franchis's theorem for one-dimensional smooth proper orbifold pairs of general type; see \cite[\S3]{CampanaMultiple}. \smallbreak

The theorem of Kobayashi--Ochiai can be made effective in the sense that one can give effective upper bounds for the number of dominant maps from a fixed variety to a fixed variety of general type; see \cite{BandmanDethloff,Heier, NaranjoPirola}. It seems reasonable to expect that one can obtain similar effective statements in the orbifold setting. \smallbreak 

One part of the Green--Griffiths--Lang conjecture predicts that every complex projective hyperbolic variety is of general type. In particular,  the theorem of Kobayashi--Ochiai suggests  that a similar finiteness statement for dominant maps should hold for projective hyperbolic varieties. Such a finiteness result for projective hyperbolic varieties was in fact already conjectured by Lang in the early seventies (see \cite{Lang1}) and proven by Noguchi \cite{Noguchi} (see also \cite{Suzuki}) in the early nineties. \smallbreak

We stress that, conjecturally, a complex projective variety is of general type if and only if it is ``pseudo-hyperbolic'', i.e., there is a proper closed subset $\Delta\subsetneq X$ such that every entire curve $\mathbb{C}\to X^{\an}$ lands in $\Delta$. The analogous finiteness statement for dominant maps to a pseudo-hyperbolic projective variety is currently not known. In particular, its extension to Campana's orbifold pairs is not known either. \smallbreak

As a straightforward application of Theorem \ref{thm1}, we prove the finiteness of the set of surjective endomorphisms of an orbifold pair of general type; see Corollary \ref{corollary:end} for a precise statement. \smallbreak 

We were first led to investigate the orbifold extension of the theorem of Kobayashi--Ochiai    in our joint work with Rousseau on rational points over number fields. We refer the reader to \cite[Theorem~1.1]{BJR} for arithmetic applications of our orbifold extension of Kobayashi--Ochiai's finiteness theorem. \smallbreak

Our proof of Theorem \ref{thm1} follows the general strategy of Kobayashi-Ochiai (and Tsushima). However, these proofs crucially rely on properties of differential forms on (log-)general type varieties (see,  for example,  \cite[Lemma~7]{Tsushima} for a key step in Tsushima's proof relying on properties of tensor powers of the sheaf of differential forms). The main difficulty in proving Theorem \ref{thm1} is that differentials for an orbifold pair $(X,\Delta)$ are not well-behaved. One may even say that there is no meaningful way to define a sheaf $\Omega^1_{(X,\Delta)}$ of orbifold differentials on $X$. On the other hand, in his seminal work on orbifold pairs \cite[\S 2]{CampanaJussieu}, Campana suggests to instead use locally free sheaves which mimick sheaves of symmetric differentials. These sheaves are abusively denoted by $S^n \Omega^p_{(X,\Delta)}$ despite the lack of existence of $\Omega^p_{(X,\Delta)}$; see Section \ref{section:sym} for a precise definition. The aforementioned key step in Tsushima's proof is then replaced by an argument involving symmetric differentials on $(X,\Delta)$; see the proof of Proposition \ref{prop0} for details.

\subsection{Conventions} 
We work over an algebraically closed field $k$.
A \emph{variety} is an integral separated scheme of finite type over $k$.
If $X$ and $Y$ are varieties, we write $X \times Y$ for $X \times_{\Spec k} Y$. 
A \emph{point} of a variety is a schematic point and need not be closed. 
If $\mathcal{L}$ is a line bundle, $D$ is a $\mathbb{Q}$-divisor, and $n$ is a natural number such that $nD$ is a $\mathbb{Z}$-divisor, we abuse notation and write $(\mathcal{L}(D))^{\otimes n}$ instead of $\mathcal{L}^{\otimes n}(nD)$.

\begin{ack}
The second-named author thanks Erwan Rousseau for many helpful discussions on Campana's theory of orbifolds. We thank Pedro N\'u\~nez for pointing out a mistake in an earlier version of the paper (see Remark \ref{remark:pedronunez}). 
\end{ack}

\section{Orbifold near-maps}
In this paper, we work with the more general notion of an orbifold near-map (Definition \ref{def:orbifold_near_map}).   This is a natural replacement of ``rational map of varieties'' in the setting of orbifold pairs.

\begin{definition}
An open subscheme $U \subseteq X$ of a variety $X$ is \emph{big} if its complement is of codimension at least two. 
\end{definition}

\begin{definition}
A rational map $X \ratmap Y$ of varieties is a \emph{near-map} if there is a big open $U \subseteq X$ such that $U \ratmap Y$ is a morphism.   
\end{definition}

Note that a rational map $X \ratmap Y$ is a near-map if and only if it is defined at all codimension one points of $X$. For example, for every normal variety $X$ and any proper variety $Y$, every rational map $X \ratmap Y$ is a near-map.

\begin{remark}\label{remark:pullback_linebundles}
If $X$ and $Y$ are varieties with $X$ locally factorial, and $f \colon X \ratmap Y$ is a near-map, we can pull back a line bundle $\mathcal{L}$ on $Y$ to a line bundle $\tilde{\mathcal{L}}$ on $X$. Indeed, while the pullback bundle $f^*\mathcal{L}$ is a priori only defined on a big open of $X$, as $X$ is locally factorial, it extends uniquely to a line bundle on all of $X$ by \cite[Proposition II.6.5(b) and Corollary II.6.16]{Hartshorne}. Since locally factorial schemes are normal, global sections of $\mathcal{L}$ also pull back to global sections of $\tilde{\mathcal{L}}$. 
\end{remark}

\begin{definition}\label{def:orbifold_near_map}
Let $(X, \Delta_X)$ be a normal orbifold and $(Y, \Delta_Y)$ be an orbifold such that $Y$ is locally factorial. Then an \emph{orbifold near-map} \[f \colon (X, \Delta_X) \ratmap (Y, \Delta_Y)\] is a near-map $f \colon X \ratmap Y$ satisfying $f(X) \nsubseteq \supp \Delta_Y$ such that, for every prime divisor $E \subseteq \supp \Delta_Y$ and every prime divisor $D \subseteq \supp f^*E$, we have $t m(D) \geq m(E)$, where $t \in \mathbb{Q}$ denotes the coefficient of $D$ in $f^*E$; this pullback is well-defined as $E$ is Cartier. As before, this is equivalent to requiring $t - 1 + \nu_D \geq t \nu_E$, where $\nu_D$ and $\nu_E$ are the coefficients of $D$ in $\Delta_X$ and of $E$ in $\Delta_Y$, respectively.
\end{definition}

Caution is advised: the composition of orbifold morphisms need not be an orbifold map. Indeed, although the condition on the multiplicities of divisor pullbacks is stable under composition, the image of the composition might be completely contained in the orbifold divisor of the target. For example, consider the morphism $\mathbb{P}^1 \to (\mathbb{P}^1, (1/2)\cdot \underline{\infty})$ given by $z \mapsto z^2$ and the inclusion of the point $\infty$ into $\mathbb{P}^1$. While both morphisms are orbifold, their composition is the inclusion of the point $\infty$ into $(\mathbb{P}^1, (1/2)\cdot \underline{\infty})$ which is not an orbifold morphism. There is however one situation in which we can compose orbifold morphisms. Namely, if the composition of two orbifold morphisms is dominant, then the composition is again an orbifold morphism. Similar composition behaviour continues to hold for orbifold near-maps. 

\begin{lemma}\label{lemma:near-map-composition}
Let $(Y, \Delta_Y)$ be an orbifold such that $Y$ is locally factorial. Let $X$ and $Z$ be locally factorial varieties. Let $f \colon X \ratmap (Y, \Delta_Y)$ be an orbifold near-map and let $g \colon Z \ratmap X$ be a near-map. If the composition $f \circ g$ exists, is defined in codimension $1$ and does not factor over $\supp \Delta_Y \subset Y$, it defines an orbifold near-map $Z \ratmap (Y,\Delta_Y)$. 
\end{lemma}
\begin{proof}
Let $U \subseteq Z$ be the open subset on which the naive composition $f \circ g$ exists, and let $V \subseteq X$ be the maximal open on which $f$ is defined. Then, by assumption, $U$ is a big open of $Z$. The morphism $f \circ g \colon U \to Y$ is the composition of the morphism of varieties $U \to V$ with the orbifold morphism $V \to (Y, \Delta_Y)$. As the image of the composition is not contained in $\supp \Delta_Y$, the composition $f \circ g \colon U \to (Y, \Delta_Y)$ is thus a morphism of orbifolds. This concludes the proof. 
\end{proof}

\begin{remark}\label{remark:pedronunez}
The assumption that the composition $f \circ g$ exists and is defined in codimension $1$ is necessary. In fact, the weaker condition  that $f \circ g$ exists and extends to a map defined in codimension $1$ does not suffice. This was overlooked in an earlier version of the paper. We give a counterexample. Let $Y$ be the blowup of $\mathbb{P}^2$ in one point $p \in Y$ and let $E \subset Y$ be the exceptional divisor. The inverse of the blowup defines an orbifold near-map $f \colon \mathbb{P}^2 \ratmap (Y, (1/2) \cdot E)$, as the pullback of $E$ along this map is trivial. Now, if $C \subset \mathbb{P}^2$ is any smooth curve passing through the blown up point $p$, the composition $C \to \mathbb{P}^2 \ratmap (Y, (1/2) \cdot E)$ is not defined at $p$. This map extends to a morphism of varieties $C \to Y$. However, this composition is no longer an orbifold morphism, as the pullback of the divisor $E$ to $C$ is given by $[p]$, which does not have the required multiplicity. 
\end{remark}

We note the following special case of Lemma \ref{lemma:near-map-composition}.

\begin{corollary}\label{corollary:near-map-restriction}
Let $X$ be a smooth variety, and let $(Y,\Delta_Y)$ be a smooth proper orbifold. Let $Z \subset X$ be a locally closed smooth subvariety and let $f \colon X \ratmap (Y,\Delta_Y)$ be a dominant orbifold near-map. If $f|_Z$ exists, is defined in codimension one, and is still dominant, then it defines an orbifold near-map $f|_Z \colon Z\ratmap (Y,\Delta_Y)$. 
\end{corollary}

In Section \ref{section:sym} and in the proof of Theorem \ref{thm1} it will be convenient to use the following notion of products for orbifold pairs.

\begin{definition}\label{def:prod}
If $(X, \Delta_X)$ and $(Y, \Delta_Y)$ are two orbifolds, then we define the \emph{product orbifold} by 
\[(X, \Delta_X) \times (Y, \Delta_Y) := (X \times Y, \Delta_X \times Y + X \times \Delta_Y).\]
\end{definition}

If $X$ and $Y$ are locally factorial, the product of orbifolds defined above satisfies the universal property of a product. More specifically, for any orbifold $(T, \Delta_T)$ and for any two orbifold morphisms 
\[\phi_X \colon (T, \Delta_T) \to (X, \Delta_X), \quad  \phi_Y \colon (T, \Delta_T) \to (Y, \Delta_Y),\] 
there is a unique orbifold morphism $\phi \colon (T, \Delta_T) \to (X, \Delta_X) \times (Y, \Delta_Y)$ such that $\phi_X = \pi_X \circ~\phi$ and $\phi_Y = \pi_Y \circ~\phi$. Indeed, it is clear that there is a morphism $\phi \colon (T, \Delta_T) \to X \times Y$, and we just have to check that this is indeed an orbifold morphism after equipping $X \times Y$ with its orbifold structure. First note that the set of closed points $t \in T$ satisfying $\phi_X(t) \notin \supp \Delta_X$ is a non-empty, hence dense, open subset of $T$. Of course, the same holds for the condition $\phi_Y(t) \notin \mathrm{supp} \Delta_Y$, so there is a closed point $t \in T$ satisfying both conditions. Thus, the image of $\phi$ is not contained in $\supp (\Delta_X \times Y + X \times \Delta_Y)$. Now let $E \subseteq \supp (\Delta_X \times Y + X \times \Delta_Y)$ be a prime divisor. Without loss of generality, we may assume that $E = E_X \times Y$ for some prime divisor $E_X \subseteq X$. Now let $D \subseteq \supp \phi_X^* E_X$ be a prime divisor, and let $r$ be its coefficient in $\phi_X^* E_X$. Since we know that $\phi_X$ is an orbifold morphism, we have $r m(D) \geq m(E_X)$. Now note that $\phi^* E = \phi_X^* E_X$, so that $r$ is also the coefficient of $D$ in $\phi^* E$. Furthermore, we have $m(E_X) = m(E)$. Thus, $r m(D) \geq m(E)$, and $\phi$ is an orbifold morphism, as desired.

\section{Families of maps}  

In this section we consider families of maps, and prove that certain conditions on the maps are either open or closed. More precisely, we show in Lemma \ref{closed_hom} that a morphism landing in a closed subscheme is a closed condition, in Lemma \ref{rat_dominant_open} that a rational function being dominant is an open condition, and in Lemma \ref{prescribed_poles_closed} that the pullback of a fixed differential form having no poles outside some fixed divisor is a closed condition.

\begin{lemma}\label{closed_hom}
Let $S$ be a scheme. Let $X\to S$ be a flat morphism whose geometric fibres are reduced. Let $Y$ and $T$ be $S$-schemes, and let $F \colon X \times_S T \to Y$ be an $S$-morphism. Then, for every closed subscheme $Z \subset Y$, the set of $t$ in $T$ such that $F_t \colon X \times_S \{t\} \to Y$ factors over $Z$ is closed in $T$.
\end{lemma}
\begin{proof}
Let $T_1\subset T$ be the set of $t$ in $T$ such that $F_t$ factors over $Z$. We consider $T_1 = \sqcup_{t\in T_1} \Spec \kappa(t)$ as an $S$-scheme. To show that $T_1$ is closed in $T$, it suffices to show that $\overline{T_1} = T_1$. We endow $\overline{T_1}\subset T$ with the reduced closed subscheme structure. The natural morphism $T_1\subset \overline{T_1}$ is dominant. Since $X\to S$ is flat, the basechange $X\times_S T_1\to X\times_S \overline{T_1}$ is (also) dominant. Since $X\times_S T_1\to Y$ factors set-theoretically through $Z$ and $X\times_S T_1\to X\times_S \overline{T_1}$ is dominant, we see that the restriction of $F$ to $X\times_S \overline{T_1}$ (also) factors set-theoretically through $Z$. However, for any $t\in \overline{T_1}$, the scheme $X\times_S \Spec \kappa(t)$ is geometrically reduced over $\kappa(t)$. In particular, $F_t$ factors (scheme-theoretically) through $Z$, as required.
\end{proof}

\begin{definition}
A rational map $X \ratmap Y$ of varieties over $k$ is \emph{separably dominant} if it is dominant and $k(Y) \subset k(X)$ is separable.
\end{definition}

\begin{lemma}\label{dominant_tangent0}
Let  $f \colon X \to Y$ be a morphism of smooth varieties. Then the following are equivalent:
\begin{enumerate}[label=(\alph*)]
\item\label{dt_dom0} The morphism $f$ is separably dominant.
\item\label{dt_tangentiso0} There is a closed point $x \in X$ such that $df_x \colon \Tangent_x X \to \Tangent_{f(x)} Y$ is surjective.
\item\label{dt_smooth0} There is a closed point $x \in X$ such that $f$ is smooth at $x$.
\end{enumerate}
\end{lemma}
\begin{proof}  
Assume \ref{dt_dom0} holds. Let $Z \subseteq X$ be the locus of points where the rank of $df_x$ is less than $\dim Y$. Since $K(Y)\subset K(X)$ is a separable field extension, the arguments used to prove \cite[Lemma~III.10.5]{Hartshorne} and \cite[Proposition~III.10.6]{Hartshorne} show that the dimension of $\overline{f(Z)}$ is less than $\dim Y$. Thus, since $f$ is dominant, this implies that $Z \neq X$. In particular, there is a closed point $x \in X$ such that the rank of $df_x$ is equal to $\dim Y$. Since $Y$ is smooth, this shows that \ref{dt_tangentiso0} holds.  

Assume \ref{dt_tangentiso0} holds. Then \cite[Proposition III.10.4]{Hartshorne} shows that \ref{dt_smooth0} holds.  

Now assume that \ref{dt_smooth0} holds. There is an open neighborhood $U \subseteq X$ of $x$ such that $f\vert_U$ is smooth. Smooth maps are flat, and flat maps of varieties are open. Hence $f(U)$ is a nonempty open of $Y$. Thus, $f$ is dominant. In particular, $f$ maps the generic point of $X$ to the generic point of $Y$. The smoothness of $f\vert_U$ also implies that $f$ is smooth at the generic point of $X$. Since generically smooth morphisms are separable, we see that \ref{dt_dom0} holds. This concludes the proof.
\end{proof}

\begin{lemma}\label{dominant_tangent}
Let $X$ and $Y$ be varieties of the same dimension. Assume that $X$ is smooth. Let $f \colon X \to Y$ be a morphism. Then the following are equivalent:
\begin{enumerate}[label=(\alph*)]
\item\label{dt_dom} The morphism $f$ is separably dominant.
\item\label{dt_tangentiso} There is a closed point $x \in X$ such that $df_x \colon \Tangent_x X \to \Tangent_{f(x)} Y$ is an isomorphism.
\end{enumerate}
\end{lemma}
\begin{proof}  
Assume \ref{dt_dom} holds. Let $U \subseteq Y$ be the locus of smooth points of $Y$. Then $U$ is a dense open of $Y$, and the restriction of $f \colon f^{-1}(U) \to U$ is still separably dominant. Thus, we may assume that $Y$ is smooth. In particular, by Lemma \ref{dominant_tangent0}, there is a closed point $x$ in $X$ such that $df_x \colon \Tangent_x X \to \Tangent_{f(x)} Y$ is surjective. Since $X$ and $Y$ are smooth of the same dimension, it follows that $df_x$ is an isomorphism.  

Assume \ref{dt_tangentiso} holds. Then, the point $f(x)$ is a regular point of $Y$. We can thus replace $Y$ by an open neighborhood of $f(x)$ and replace $X$ by the preimage of that. The assumption on the dimensions continues to hold, so we may assume that $Y$ is smooth. Consequently, by Lemma \ref{dominant_tangent0}, the morphism $f$ is separably dominant.
\end{proof}

\begin{definition}
If $X$, $Y$ and $T$ are varieties, we say that a rational map $f \colon X \times T \ratmap Y$ is a \emph{relative rational map (over $T$)} if the maximal open subset $U \subseteq X \times T$ on which $f$ is defined has nonempty intersection with every closed fiber $X_t := X \times \{t\}$. In other words, it is a family of rational maps $f_t \colon X \ratmap Y$ parametrized by the variety $T$. 
\end{definition}

\begin{lemma}\label{rat_dominant_open}
Let $X$ and $Y$ be varieties of the same dimension and let $T$ be any variety. Let $F \colon X \times T \ratmap Y$ be a relative rational map. Then the locus of $t \in T$ such that $F_t \colon X \ratmap Y$ is separably dominant is open in $T$. 
\end{lemma}
\begin{proof}
Replacing $X$ by an alteration if necessary, we may assume that $X$ is smooth \cite[Theorem~4.1]{deJongAlterations}. Let $U \subseteq X \times T$ be the maximal open subset on which $F$ is defined. The map $F$ then induces a morphism of $T$-schemes $G \colon U \to Y \times T$.  

We claim that if $(x,t) \in U$ is any closed point, the differential $dG_{(x,t)}$ is an isomorphism if and only if the differential of the rational map $F_t \colon X \ratmap Y$ is an isomorphism at $x \in X$. To see this, note that the tangent spaces of $X \times T$ and $Y \times T$ are the products of the tangent spaces of the factors. Furthermore, the component of $dG_{(x,t)}$ which maps $\Tangent_t T \to \Tangent_{F(x,t)} X$ is the zero map, and the component which maps $\Tangent_t T \to \Tangent_t T$ is just the identity. Lastly, the component of $dG_{(x,t)}$ mapping $\Tangent_x X \to \Tangent_{F_t(x)}$ is just $dF_t$. This implies the claim.  

Let $t \in T$ be a closed point. By Lemma \ref{dominant_tangent}, $F_t$ is separably dominant if and only if there is a closed point $x \in X$ such that the differential $dF_{t,x}$ is an isomorphism. By the previous claim, this happens if and only if there is a closed point $x \in X$ such that $dG_{(x,t)}$ is an isomorphism.  

Let $V \subseteq U$ be the set of all points at which the differential of $G$ is an isomorphism. The set $V$ is open in $U$, hence open in $X \times T$. The map $X \times T \to T$ is flat, hence open. Thus, the projection of $V$ to $T$ is an open subset of $T$. By the previous paragraph, this is exactly the set of $t \in T$ for which $F_t$ is separably dominant, so we are done.
\end{proof}

\begin{remark}
The locus of (separably) dominant maps is not necessarily closed in $T$ (this seems to have been overlooked in \cite[Proposition~6.1]{IwanariMoriwaki}).  
Consider, for example, the map $\mathbb{P}^1 \times \mathbb{P}^1 \ratmap \mathbb{P}^1$ given by $(x,t)\mapsto xt$. Its indeterminacy locus consists of the two points $(0,\infty)$ and $(\infty,0)$, so that this is indeed a relative rational map. If we fix any value $t \in \mathbb{P}^1 \setminus \{0,\infty\}$, the resulting rational map $\mathbb{P}^1 \ratmap \mathbb{P}^1$ is an isomorphism, and thus separably dominant. For $t \in \{0, \infty\}$, the resulting map is constant.  Thus,  the locus where the map is separably dominant is $\mathbb{P}^1 \setminus \{0,\infty\} \subseteq \mathbb{P}^1$; this is open but not closed. 
\end{remark}

The following statement is a purely algebraic result which we will use in the proof of Lemma \ref{prescribed_poles_closed} below. Recall that if $M$ is an $R$-module, and $f \in R$ is an element, the element $f$ is called $M$-regular if the morphism $M \to M$, $m \mapsto fm$ is injective. In the special case that $M = S$ is an $R$-algebra, this is equivalent to $f$ being a nonzerodivisor in $S$. If $S$ is additionally assumed to be reduced and noetherian, this in turn is equivalent to asking that $\Vanish(f) \subseteq \Spec S$ has codimension $\geq 1$ (where the empty set is considered to have codimension $\infty$). 
 
\begin{lemma}\label{matsumura_flatness}
Let $R$ be a noetherian ring, let $S$ be a noetherian $R$-algebra, let $f \in S$. Let $M$ be a finitely generated $S$-module. Assume that $M$ is flat over $R$ and assume that for every maximal ideal $\mathfrak{m} \subseteq S$, the element $f$ is $M/(\mathfrak{m} \cap R)M$-regular. Then $f$ is $M$-regular and $M/fM$ is flat over $R$.
\end{lemma}
\begin{proof}
See \cite[§20F Corollary 1]{Matsumura}.
\end{proof}

The following lemma and proof are essentially due to Tsushima \cite[Lemma 5]{Tsushima}. 
Before starting the proof, we briefly discuss extending relative rational maps. Let $G \colon X \times T \ratmap Y$ be a relative rational map with $X$, $T$ normal varieties, and $Y$ a proper variety. Then, by properness of $Y$ and normality of $X \times T$, the rational map $G \colon X \times T \ratmap Y$ extends to a morphism $U \to Y$ on a maximal open set $U \subseteq X \times T$ with complement of codimension at least $2$. However, for any given closed point $t \in T$, it might still happen that $U_t := U \cap X_t \subseteq X$ has a complement of codimension $1$. While the restriction of $G$ to $U_t$ will extend to a rational map $X \ratmap Y$ defined in codimension $1$, this extension will in general not be compatible with $G \colon X \times T \ratmap Y$. 

\begin{lemma}\label{prescribed_poles_closed}
Let $X$ and $Y$ be smooth projective varieties of the same dimension $n$, and let $T$ be a variety. Let $G \colon X \times T \ratmap Y$ be a relative rational map over $T$. Let $D_X$ be an effective divisor on $X$ and let $D_Y$ be a divisor on $Y$. Let $m \geq 0$. Let $\omega \in \Gamma(Y, \omega_Y^{\otimes m}(D_Y))$.  Assume that,  for every closed point $t \in T$, the rational map $G_t \colon X \ratmap Y$ is separably dominant. Then, the set $T_\omega$ of $t \in T$ such that $G_t^*\omega $ lies in $ \Gamma(X, \omega_X^{\otimes m}(D_X))$ is closed in $T$.
\end{lemma}
\begin{proof}
The case $\omega = 0$ is clear, so suppose $\omega \neq 0$. By replacing $T$ with an alteration if necessary, we may assume that $T$ is smooth.  

Consider the pullback form $G^*\omega$, and note that it defines a rational section of the vector bundle $(\Omega^n_{X \times T})^{\otimes m}$. For any closed point $t \in T$, we pullback $G^*\omega$ along the inclusion $\iota_t \colon X := X\times\{t\} \subseteq X \times T$ to get the form $\iota_t^*G^*\omega = G_t^*\omega$.
Let $E$ and $F$ be the divisors of zeroes and poles of $G^*\omega$, respectively. For $t$ in $T$, we define $E_t := E \cap X_t$ and $F_t := F \cap X_t$. Since $G_t$ is separably dominant and $\omega \neq 0$, we have that, for every $t$ in $T$, $E_t$ and $F_t$ are (possibly trivial) effective divisors in $X_t$. 
Note that whenever $G_t$ is defined in codimension one, we have that $E_t$ (resp. $F_t$) is the divisor of zeroes (resp. poles) of $G_t^*\omega$. On the other hand, if $G_t$ is not defined at all points of codimension $1$, it may happen that $E_t$ and $F_t$ are strictly bigger than the divisor of zeroes and poles, respectively.

We now prove the result by induction on $\dim(T)$. The case $\dim(T)=0$ is clear. Consider the set $S$ of all $t$ in $T$ such that $\dim(E_t \cap F_t) = n-1$. By semicontinuity of fiber dimension, $S$ is closed in $T$ (since $\dim(E_t \cap F_t) > n-1$ cannot occur). The condition $t \in S$ implies that $G_t$ cannot be defined in codimension $1$. Otherwise the form $G_t^*\omega$ would have to have both a pole and a zero along the codimension $1$ subset $E_t \cap F_t$, which is absurd. Since $G$ is defined at all points of codimension $1$, we see $\dim(S)<\dim(T)$. By the inductive hypothesis, it follows that $S_\omega=S \cap T_\omega$ is closed.   

We now show that $F\to T$ is flat using Lemma \ref{matsumura_flatness}. First, note that it is locally cut out by the vanishing of a single equation given by the denominator of $G^*\omega$. Furthermore, the morphism $X \times T \to T$ is flat and, for every closed point $t$ in $T$, the scheme-theoretic fiber $F_t$ of $F \to T$ is a divisor of $X_t$. 
Thus, we conclude that $F \to T$ is flat by Lemma \ref{matsumura_flatness}. 
 
In particular, there is a morphism $T \to \mathrm{Hilb}(X)$ representing the family $(F_t)_{t \in T}$, where $\mathrm{Hilb}(X)$ is the Hilbert scheme of $X$ over $k$. Now, as there are only finitely many effective divisors with the property of being $\leq D_X$, the set of such divisors form a (finite) closed subscheme of $\mathrm{Hilb}(X)$. It follows that $F_t \leq D_X$ is a closed condition on $t$. 

For $t$ in $T$, the condition $F_t\leq D_X$ implies that $G_t^* \omega \in \Gamma(X,\omega_X^{\otimes m}(D_X))$. Moreover, outside the set $S$ (defined above), the condition $G_t^*\omega \in \Gamma(X,\omega_X^{\otimes m}(D_X))$ is equivalent to $F_t \leq D_X$.
Thus, a point $t \in T$ lies in $T_\omega$ if and only if we have $t \in S \cap T_\omega$ or $F_t \leq D_X$. As $S\cap T_\omega$ is closed in $T$ and the set of $t$ in $T$ with $F_t\leq D_X$ is closed in $T$, this concludes the proof.
\end{proof}

\section{Symmetric differentials on orbifolds}\label{section:sym}
In this section we collect some statements regarding the sheaf of symmetric differentials on an orbifold. We start by recalling their definition, first given by Campana in \cite[Section 2.5]{CampanaJussieu}.

\begin{definition}\label{def:symmetricdifferentials}
Let $(X, \Delta)$ be a smooth orbifold. Let $n, p \geq 0$ be natural numbers. The sheaf of \emph{symmetric differentials}, written $S^n \Omega^p_{(X,\Delta)}$, is the locally free subsheaf of $\Sym^n \Omega^p_X(\log \ceil{\Delta})$ which is étale-locally generated by the following elements:
$$ x^{\ceil{k/m}} \bigotimes_{i=1}^n \frac{dx_{J_i}}{x_{J_i}} $$
Here, the following notation was used:
\begin{itemize}
\item $x_1,...,x_{\dim(X)}$ are a set of local coordinates which exhibit $\Delta$ in normal crossing form.
\item The $J_i$ are subsets of $\{1,...,\dim(X)\}$ of size $p$.
\item $dx_{J_i} := \bigwedge_{j \in J_i} dx_j$ and $x_{J_i} := \prod_{j \in J_i} x_j$
\item $k$ is a tuple of $\dim(X)$ integers, where the $j$-th entry counts the number of occurences of $j$ in the $J_i$.
\item $m$ is a tuple of $\dim(X)$ integers, where the $j$-th entry is the multiplicity of the coordinate $x_j$ in $\Delta$.
\item $x^{\ceil{k/m}} := \prod_{j=1}^{\dim(X)} x_j^{\ceil{k_j/m_j}}$
\end{itemize}
\end{definition}

\begin{definition}
Let  $(X, \Delta_X)$ be a normal orbifold. Choose a big open $U \subseteq X$ such that $U$ is smooth and $\Delta_U := \Delta_X \cap U$ has strict normal crossings. We define $S^n \Omega^p_{(X, \Delta_X)} := \iota_* S^n \Omega^p_{(U, \Delta_U)}$, where $\iota \colon U \to X$ is the inclusion.
\end{definition}

For smooth proper varieties $X$ without any orbifold structure, the sheaves $S^n \Omega^p_{(X,0)}$ defined this way coincide with the usual symmetric powers of the module of differentials $\Sym^n \Omega^p_X$. More generally, if $(X, \Delta)$ is an orbifold where all multiplicities in $\Delta$ are equal to $1$ or $\infty$, the sheaves $S^n \Omega^p_{(X,\Delta)}$ defined above coincide with the symmetric powers of the module of log differentials $\Sym^n \Omega^p_X(\log \Delta)$. However, in general, the sheaves $S^n \Omega^p_{(X,\Delta)}$ are not the symmetric powers of any coherent sheaf (so that calling them symmetric differentials is a significant abuse of language). The main use of $S^n \Omega^p$ for us comes from the fact that these sheaves behave nicely when they are pulled back by orbifold morphisms.
 
\begin{lemma}\label{pullback_of_forms}
If $f \colon (X, \Delta_X) \to (Y, \Delta_Y)$ is a morphism of smooth orbifolds and $n, p \geq 0$, then pullback of differential forms induces a morphism $f^*S^n \Omega^p_{(Y, \Delta_Y)} \to S^n\Omega^p_{(X,\Delta_X)}$.
\end{lemma}
\begin{proof}
Campana shows this when $k=\mathbb{C}$ using computations in the analytic topology, see \cite[Proposition 2.11]{CampanaJussieu}. His arguments adapt to positive characteristic, as we show now.  

We have a morphism of sheaves \[f^* \Sym^n \Omega^p_Y(\log \ceil{\Delta_Y}) \to \Sym^n \Omega^p_X(\supp (\Delta_X + f^* \Delta_Y)).\]
As $f^*S^n \Omega^p_{(Y, \Delta_Y)}$ (resp. $S^n\Omega^p_{(X,\Delta_X)}$) is a subsheaf of the source (resp. the target) of this morphism, we may argue \'etale-locally around a fixed point $\eta$. As the sheaves involved are locally free, we may and do assume that  $\eta \in X$ is of codimension $1$ (except for in the trivial situation in which $X$ is zero-dimensional).

Locally around $\eta$, the sheaf $f^* S^n \Omega^p_{(Y,\Delta_Y)}$ is generated by the pullbacks of the local generators of $S^n \Omega^p_{(Y,\Delta_Y)}$ around $f(\eta)$ (see Definition \ref{def:symmetricdifferentials}). Let $\omega \in S^n \Omega^p_{(Y, \Delta_Y)}$ be such a generator. Let $\widetilde{Y} \to Y$ be a connected étale neighborhood of $f(\eta)$ such that   
\begin{enumerate}
\item $\widetilde{Y}$ is an étale open of $\mathbb{A}^d$ with $d = \dim Y$,
\item $\Delta_Y$ is in normal crossing form (i.e., $\Delta_Y$ is given by the pullback of $x_1\cdot \ldots \cdot x_\ell =0$ in $\mathbb{A}^d$ for some $\ell\geq 0$), 
\item $f(\eta)$ specializes to the origin of $\mathbb{A}^d$, and
\item $\omega$ has the form described in Definition \ref{def:symmetricdifferentials}.  
\end{enumerate}

There is a connected étale neighborhood $\widetilde{X} \to X$ of $\eta$ such that $f\vert_{\widetilde{X}}$ factors over $\widetilde{Y}$ and such that $\Delta_X$ is in normal crossings form. Since the induced morphism $f \colon (\widetilde{X}, \widetilde{X} \times_X \Delta_X) \to (\widetilde{Y}, \widetilde{Y} \times_Y \Delta_Y)$ is (still) orbifold, we may replace $(X, \Delta_X)$ by $(\widetilde{X}, \widetilde{X} \times_X \Delta_X)$ and $(Y, \Delta_Y)$ by $(\widetilde{Y}, \widetilde{Y} \times_Y \Delta_Y)$.

As $Y$ is an \'etale open of $\mathbb{A}^d$, we obtain a map $X \to \mathbb{A}^d$ given by $d$ maps $f_1,...,f_d \colon X \to \mathbb{A}^1$. For $i=1,\ldots,d$, we let $m_i$ denote the multiplicity of the (pullback of the) prime divisor $\{y_i=0\}$ in $\Delta_Y$. Since $f(X)$ is not contained in $\Delta_Y$, we know that whenever $m_i > 1$, the function $f_i$ is not identically zero. Viewing $f_i$ as an element of the DVR $\mathcal{O}_{X, \eta}$, we decompose it as $f_i = t^{\nu_i} g_i$ with $t$ a uniformizer and $g_i(\eta) \neq 0$. Since $f$ is an orbifold morphism, for any $i$ with $\nu_i \neq 0$, we have $\nu_i \geq \ceil{m_i/m_\eta}$, where $m_\eta$ is the multiplicity of the divisor $\eta$ in $\Delta_X$.

Let $J \subseteq \{1,...,d\}$ be a $p$-element subset and consider the rational $p$-form $dy_J/y_J$ on $Y$. If none of the functions $f_i$ with $i \in J$ vanish along $\eta$, the pullback $f^*(dy_J/y_J)$ has no pole at $\eta$. If such an $i \in J$ exists, then $f^*(dy_J/y_J)$ has a pole of order at most $1$. Since we can always write $f^*(dy_J/y_J) = (dt/t) \wedge u + v$, where $u$ is a $(p-1)$-form with no pole at $\eta$ and $v$ a $p$-form with no pole at $\eta$, the pullback of $\omega$ is given by   
\[ f^*\omega = \prod_{i=1}^d (f_i)^{\ceil{k_i/m_i}} \bigotimes_{\alpha=1}^n \left( (\frac{dt}{t} \wedge u_\alpha) + v_\alpha  \right).\]
Here, as before, $u_\alpha$ and $v_\alpha$ are forms with no pole at $\eta$. We can write the tensor product of sums as a sum of tensor products. When doing this, the order at $\eta$ of each summand occuring in such a rewriting is at least
$$ \left(\sum_{i=1}^d \ceil{\frac{k_i}{m_i}} \nu_i\right) - k_t, $$
where $k_t$ counts the number of $((dt/t) \wedge u_\alpha)$-factors occuring in that summand (as opposed to $v_\alpha$-factors). Note that $k_t \leq \sum k_i$, where the sum runs over those $i$ for which $\nu_i \geq 0$. Using our estimate $\nu_i \geq \ceil{m_i/m_\eta}$ from before, we obtain that the order at $\eta$ of each summand is at least
$$ \left(\sum_{\substack{i=1 \\ \nu_i \neq 0}}^d \ceil{\frac{k_i}{m_\eta}}\right) - k_t \geq -k_t + \ceil{\frac{k_t}{m_\eta}}, $$
so the pole at $\eta$ is at most of the order we allow for elements of $S^n \Omega^p_{(X, \Delta_X)}$. Hence $f^*\omega \in S^n \Omega^p_{(X, \Delta_X)}$, which concludes the proof.
\end{proof}

If $X$ is a smooth variety of dimension $n$, the sheaf $S^1 \Omega^n_{(X,0)}$ is just the canonical sheaf $\omega_X$. Thus, one might guess that for a smooth orbifold $(X, \Delta)$, the sheaf $S^1 \Omega^n_{(X,\Delta)}$ should correspond to a line bundle related to the $\mathbb{Q}$-divisor $K_X + \Delta$. Of course, naively formulated like this, this guess does not really make sense, since $K_X + \Delta$ is not a $\mathbb{Z}$-divisor and hence does not correspond to any line bundle. However, as we show now, the intuition can be saved. (Recall our convention that, for $\mathcal{L}$ a line bundle, $D$ a $\mathbb{Q}$-divisor, and $n$ a natural number such that $nD$ is a $\mathbb{Z}$-divisor, we write $(\mathcal{L}(D))^{\otimes n}$ instead of $\mathcal{L}^{\otimes n}(nD)$.)

\begin{lemma}\label{symmetric_to_canonical}
Let $(X, \Delta)$ be a smooth orbifold of dimension $n$ and let $N$ be a natural number such that $N \Delta$ is a $\mathbb{Z}$-divisor. Then $S^N \Omega^n_{(X, \Delta)} \cong \omega_X(\Delta)^{\otimes N}$.
\end{lemma}
\begin{proof}
For the sheaf of log-differentials, we have $\Omega^n_X(\log \ceil{\Delta}) = \omega_X(\ceil{\Delta})$ (see \cite[§11.1]{Iitaka}). It follows that $\Sym^N \Omega^n_X(\log \ceil{\Delta}) = \Sym^N \omega_X(\ceil{\Delta})$. Since symmetric powers of line bundles agree with tensor powers, it follows that $\Sym^N \omega_X(\ceil{\Delta}) = \omega_X(\ceil{\Delta})^{\otimes N}$. Thus, $S^N \Omega^n_{(X, \Delta)}$ is by construction a locally free subsheaf of $\omega_X(\ceil{\Delta})^{\otimes N}$. More precisely, we see that locally around a point $p \in X$, it is the subsheaf generated by the single element
$$ x_1^{N/m_1}...x_n^{N/m_n} \bigotimes_{l=1}^N \frac{dx_1 \wedge dx_2 ... \wedge dx_n}{x_1x_2...x_n} $$
where $x_1,...,x_n$ are a set of normal crossing coordinates for $\Delta$, and $m_i$ denotes the multiplicity of the coordinate $x_i$ in the orbifold divisor. The subsheaf generated by this element is equal to $\omega_X(\Delta)^{\otimes N}$ in some neighborhood of $p$. The claim follows since $p$ was arbitrary. 
\end{proof}

\begin{corollary} \label{pullback_of_forms2}  
Let $f \colon (X, \Delta_X) \ratmap (Y, \Delta_Y)$ be a near-map of orbifolds with $n := \dim X = \dim Y$. Assume that $X$ is smooth and that $(Y, \Delta_Y)$ is smooth. Let $N$ be a natural number such that $N \Delta_Y$ is a $\mathbb{Z}$-divisor. Then there is an induced pullback morphism $f^* \omega_Y(\Delta_Y)^{\otimes N} \to \omega_X(\Delta_X)^{\otimes N}$ of locally free sheaves on $X$.   
\end{corollary}
\begin{proof}
By Lemma \ref{symmetric_to_canonical}, we have $\omega_Y(\Delta_Y)^{\otimes N}= S^N \Omega^n_{(Y, \Delta_Y)}$. Similarly, as $n = \dim X = \dim Y$, we have $\omega_X(\Delta_X)^{\otimes N} = S^N \Omega^n_{(X, \Delta_X)}$, as this equality holds over a big open and both sides are line bundles. 
By Lemma \ref{pullback_of_forms}, we get a morphism $f^* S^N \Omega^n_{(Y, \Delta_Y)} \to S^N \Omega^n_{(U, \Delta_X \cap U)}$ of sheaves on $U$, where $U \subseteq X$ denotes the intersection of the domain of definition of $f$ with the snc locus of $(X, \Delta_X)$. Since $U \subseteq X$ is a big open, by Remark \ref{remark:pullback_linebundles}, the line bundle $f^*\omega_Y(\Delta_Y)^{\otimes N}$ on $U$ extends to a line bundle on $X$. Furthermore, as the morphism of locally free sheaves extends as well by Hartogs, this concludes the proof.
\end{proof}

If $X$ and $T$ are smooth varieties, and $\pi_X \colon X \times T \to X$ and $\pi_T \colon X \times T \to T$ denote the canonical projections, we get a direct sum composition for the Kähler differentials:
$$ \Omega^1_{X \times T} \cong \pi_X^* \Omega^1_X \oplus \pi_T^* \Omega^1_T $$
Passing to exterior powers, and noting that taking exterior powers commutes with taking pullbacks, we retain such a direct sum decomposition, although it gets slightly more involved:
$$ \Omega^m_{X \times T} \cong \bigoplus_{i=0}^m \pi_X^* \Omega^i_X \otimes \pi_T^* \Omega^{m-i}_T $$
Lastly, if $A$ and $B$ are modules over any commutative ring, we have the following direct sum decomposition for the symmetric powers:
$$ \Sym^n(A \oplus B) \cong \bigoplus_{i=0}^n (\Sym^i A \otimes \Sym^{n-i} B) $$
By combining the two previous lines, we obtain that $\Sym^n \pi_X^* \Omega^m_X$ is a direct summand of $\Sym^n \Omega^m_{X \times T}$. Hence we get an idempotent endomorphism of $\Sym^n \Omega^m_{X \times T}$ which projects an element into that summand. Furthermore, if $t \in T$ is any closed point and $\iota_t \colon X = X \times \{t\} \subseteq X \times T$ is the inclusion, then the pullback map $\iota_t^* \Sym^n \Omega^m_{X \times T} \to \Sym^n \Omega^m_X$ factors over that projection. We now prove the analogous result for orbifolds (see Definition \ref{def:prod} for the notion of a product of orbifold pairs).

\begin{lemma}\label{direct_summand}
Let $(X, \Delta_X)$ and $(T, \Delta_T)$ be normal orbifolds. Let $\pi_X$ denote the canonical projection $X \times T \to X$. Then for all natural numbers $N$ and $m$, the sheaf $\pi_X^* S^N \Omega^m_{(X,\Delta_X)}$ is a direct summand of $S^N \Omega^m_{(X, \Delta_X) \times (T,\Delta_T)}$.
\end{lemma}
\begin{proof}
First assume that the result is true in the case of the smooth orbifolds. Let $X^o \subseteq X$ and $T^o \subseteq T$ denote the snc loci, and note that these are big opens. Then, since forming the pushforward commutes with forming direct sums, we see that $\iota_{(X^o \times T^o)*}\pi_{X^o}^*S^N \Omega^m_{(X^o, \Delta_{X^o})}$ is a direct summand of $S^N \Omega^m_{(X, \Delta_X) \times (T, \Delta_T)}$. Since $\iota_{(X^o \times T^o)*} \pi_{X^o}^* = \pi_X^* \iota_{X^o*}$ for locally free sheaves, the desired result follows. 

By the above, we may assume that $(X, \Delta_X)$ and $(T, \Delta_T)$ are smooth.
Assume first that $\Delta_X$ and $\Delta_T$ are $\mathbb{Z}$-divisors, i.e. that all multiplicities are either $1$ or $\infty$. In this case, we have $S^N \Omega^m_{(X,\Delta_X)} = \Sym^N \Omega^m_X(\log \Delta_X)$. The latter is a genuine symmetric power of an exterior power of $\Omega^1_X(\log \Delta_X)$. Notice that the decomposition
$$ \Omega^1_{X \times T}(\log (\Delta_X \times T + X \times \Delta_T)) = \pi_X^* \Omega^1_X(\log \Delta_X) \oplus \pi_T^* \Omega^1_T(\log \Delta_T) $$
is still valid. Thus, the discussion of the previous paragraph applies, proving the result in this case.

Finally, we treat the general case in which $\Delta_X$ and $\Delta_T$ are not $\mathbb{Z}$-divisors. By definition, the sheaf $S^N \Omega^n_{(X,\Delta_X)}$ is a subsheaf of
\[ \Sym^N \Omega^m_{X}(\log \ceil{\Delta_X}), \] 
and, similarly, the sheaf $S^N \Omega^m_{(X,\Delta_X)\times (T,\Delta_T)}$ is a subsheaf of
\[ \Sym^N \Omega^m_{X\times T}(\log \ceil{\Delta_X\times T + X\times \Delta_T}). \]
By the previous paragraph we know that the morphism
\[ \pi_X^* \Sym^N \Omega^m_{X}(\log \ceil{\Delta_X}) \to \Sym^N \Omega^m_{X\times T}(\log \ceil{\Delta_X \times T + X \times \Delta_T}) \]
is injective and has a retract. Since the projection map $\pi_X$ is orbifold, it follows from  Lemma \ref{pullback_of_forms} that the above injection sends the subsheaf $\pi_X^*S^N\Omega^m_{(X,\Delta_X)}$ to the subsheaf $S^N\Omega^m_{(X,\Delta_X)\times (T,\Delta_T)}$. To prove the claim, it thus suffices to show that the retraction also respects these subsheaves. This can be checked locally, and it suffices to consider the generators. This can be done very explicitly.

Indeed, let $(x,t) \in X \times T$ be any closed point, let $dx_1,...,dx_n$ be local coordinates for $X$ around $x$ which exhibit $\Delta_X$ in normal crossings form, and let $dt_1,...,dt_r$ be local coordinates for $T$ around $t$ which exhibit $\Delta_T$ in normal crossings form. Then $dx_1,...,dx_n, dt_1,..., dt_r$ are local coordinates for $X \times T$ exhibiting its orbifold divisor in normal crossings form. Let $\omega$ be a local generator of $S^N \Omega^m_{(X, \Delta_X) \times (T,\Delta_T)}$ around $(x,t)$. If $\omega$ contains any factors containing a $dt_i$, the pullback $\iota_t^* \omega$ will be identically zero. Thus, it remains to consider the case where the only differentials appearing in $\omega$ are products of $dx_i$ terms. Pulling back such a generator of $S^N\Omega^m_{(X,\Delta_X)\times (T,\Delta_T)}$ along $\iota_t$ yields a (formally identical) generator of $S^N \Omega^m_{(X,\Delta_X)}$. This finishes the proof.
\end{proof}

\begin{lemma}\label{factor_over_projection}
Let $(X, \Delta_X)$ be a normal orbifold and let $(T, \Delta_T)$ be a smooth orbifold. For a closed point $t \in T \setminus \Delta_T$, write $\iota_t \colon X = X \times \{t\} \subseteq X \times T$. Then, for any natural numbers $m$ and $N$, the pullback map $\iota_t^* S^N \Omega^m_{(X, \Delta_X) \times (T, \Delta_T)} \to S^N \Omega^m_{(X, \Delta_X)}$ factors over the projection to $\iota_t^* \pi_X^* S^N \Omega^m_{(X,\Delta_X)}$.
\end{lemma}
\begin{proof}
We first show that the pullback map considered in the statement is well-defined. To do so, let $X^o\subseteq X$ be the snc locus of $(X,\Delta_X)$ and note that $X^o$ is a big open of $X$. The morphism $\iota_t \colon X \to X\times T$ induces an orbifold morphism $X^o\to X^o\times T$ (i.e., $\iota_t$ is orbifold over the snc locus). In particular, by Lemma \ref{pullback_of_forms}, there is an induced morphism $\iota_t^* S^N \Omega^m_{(X^o, \Delta_{X^o}) \times (T, \Delta_T)} \to S^N \Omega^m_{(X^o, \Delta_{X^o})}$. Since the sheaves involved are reflexive, this pullback map extends to  a pullback map  $\iota_t^* S^N \Omega^m_{(X, \Delta_X) \times (T, \Delta_T)} \to S^N \Omega^m_{(X, \Delta_X)}$.

To see that the pullback map factors as claimed, note that $\pi_X \circ \iota_t = \id_X$ is the identity morphism, so that in fact $\iota_t^* \pi_X^* S^N \Omega^m_{(X,\Delta_X)} = S^N \Omega^m_{(X,\Delta_X)}$. This concludes the proof.
\end{proof}

\section{Kobayashi--Ochiai's theorem for orbifold pairs}

In this section, we prove the finiteness theorem for dominant maps into a smooth orbifold of general type $(Y, \Delta_Y)$. The first step of the proof is to show that given a dominant morphism $f \colon (X, \Delta_X) \to (Y, \Delta_Y)$, we can recover $f$ from its induced map on global sections of the canonical bundles $\omega_Y(\Delta_Y)^{\otimes N}(Y) \to \omega_X(\Delta_X)^{\otimes N}(X)$ for sufficiently large $N$, where $N$ only depends on $(X, \Delta_X)$ and $(Y, \Delta_Y)$ but not on $f$. This allows us to shift the focus from studying dominant morphisms to studying linear maps $\omega_Y(\Delta_Y)^{\otimes N}(Y) \to \omega_X(\Delta_X)^{\otimes N}(X)$ satisfying certain conditions.
 
To state the next lemma, we introduce some terminology. We call a line bundle \emph{very big} if the rational map to projective space induced by its global sections is birational onto its image. Note that every big line bundle has a tensor power which is very big. Of course, a very ample line bundle is very big. Also, if $V$ is a vector space, the projective space $\mathbb{P}(V)$ parametrizes subspaces of \emph{codimension} $1$.

\begin{lemma}\label{tsushimalemma2}
Let $X$ and $Y$ be projective varieties. Assume that $X$ is locally factorial. Let $\mathcal{L}_X$ and $\mathcal{L}_Y$ be line bundles on $X$ and $Y$ respectively. Assume that $\mathcal{L}_X$ is very big and that $\mathcal{L}_Y$ is very ample. Consider the following set:
$$ S = \{ (f, \phi)~\vert~f \colon X \ratmap Y~\text{dominant and}~\phi \colon f^*\mathcal{L}_Y \to \mathcal{L}_X~\text{injective} \} $$
If $(f,\phi)$ and $(g,\psi)$ have the same image under the composed map of sets
$$ S \to \Hom(\mathcal{L}_Y(Y), \mathcal{L}_X(X))\setminus\{0\} \to \{\text{rational maps from }\mathbb{P}(\mathcal{L}_X(X))\text{ to }\mathbb{P}(\mathcal{L}_Y(Y)) \}, $$
then $f=g$.
\end{lemma}
\begin{proof}
Before starting the proof, we note that the set $S$ is well-defined by Remark \ref{remark:pullback_linebundles}.

Let $(f, \phi)$ and $(g, \psi)$ be elements of $S$ which induce the same rational map
$$\gamma \colon \mathbb{P}(\mathcal{L}_X(X)) \ratmap \mathbb{P}(\mathcal{L}_Y(Y)).$$
By our assumptions on the line bundles, the space $\mathbb{P}(\mathcal{L}_X(X))$ contains a birational copy $\overline{X}$ of $X$ and $\mathbb{P}(\mathcal{L}_Y(Y))$ contains $Y$. The following square commutes whenever the compositions are defined:
\begin{equation*} \begin{tikzcd}
X \ar[r, densely dashed] \ar[d, densely dashed] & Y \ar[d] \\  \mathbb{P}(\mathcal{L}_X(X)) \ar[r, densely dashed, "\gamma"] & \mathbb{P}(\mathcal{L}_Y(Y))
\end{tikzcd} \end{equation*}
Here, the upper horizontal arrow can be either $f$ or $g$. Note that the indeterminacy locus of $\gamma$ is a linear subspace, and that $\overline{X}$ is not contained in any proper linear subspace of $\mathbb{P}(\mathcal{L}(X))$. Hence $\gamma$ is defined on some open of $\overline{X}$ and the commutativity of the diagram above implies that $\gamma$ sends $\overline{X}$ to $Y$. So we get a rational map $\overline{X} \ratmap Y$. The composition $X \ratmap \overline{X} \ratmap Y$ is equal to both $f$ and $g$ whenever it is defined, showing that $f=g$ on a dense open subset. As $Y$ is separated, this implies that $f=g$ everywhere. 
\end{proof}

We now have all the prerequisite results needed for our proof of the announced theorem. We follow the general proof strategy of \cite{Tsushima}.  

\begin{proposition}\label{prop0}
Let $(X, \Delta_X)$ and $(Y, \Delta_Y)$ be proper orbifolds. Assume that $X$ is smooth and that $(Y, \Delta_Y)$ is smooth and of general type. If $\dim X=\dim Y$, then there are only finitely many separably dominant near-maps $(X, \Delta_X) \ratmap (Y, \Delta_Y)$. 
\end{proposition}
\begin{proof}
If there are no separably dominant near-maps from $(X,\Delta_X)$ to $(Y,\Delta_Y)$, then we are done. Otherwise, let $f \colon (X, \Delta_X) \ratmap (Y, \Delta_Y)$ be a separably dominant near-map. By Corollary \ref{pullback_of_forms2}, for $N \in \mathbb{N}$ sufficiently divisible, we get an induced morphism of line bundles
$f^* (\omega_Y(\Delta_Y))^{\otimes N} \to \omega_X(\Delta_X)^{\otimes N}$.
Since $f$ is separably dominant, the morphism of line bundles $f^* (\omega_Y(\Delta_Y))^{\otimes N} \to \omega_X(\Delta_X)^{\otimes N}$ is non-zero, hence injective. This implies that $\omega_X(\Delta_X)^{\otimes N}$ is a big line bundle. Increasing $N$ if necessary, we can thus assume that all of the following hold:

\begin{itemize}
\item $\omega_X(\Delta_X)^{\otimes N}$ and $\omega_Y(\Delta_Y)^{\otimes N}$ are well-defined line bundles.
\item The line bundle $\omega_X(\Delta_X)^{\otimes N}$ is very big.
\item There is an effective divisor $C \subseteq Y$ such that $(\omega_Y(\Delta_Y)^{\otimes N})(-C)$ is very ample. 
\end{itemize}

We fix the integer $N$ and the effective divisor $C \subseteq Y$ from the last bullet point. We define $V_X := \Gamma(X, \omega_X(\Delta_X)^{\otimes N})$ and $V_Y := \Gamma(Y, \omega_Y(\Delta_Y)^{\otimes N}(-C))$. Note that we obtain a closed immersion $\iota_Y \colon Y \to \mathbb{P}(V_Y)$ and a rational map $\iota_X \colon X \ratmap \mathbb{P}(V_X)$ which is birational onto its scheme-theoretic image $\overline{X}$. For every dominant near-map $f$, we get an induced vector space morphism $f^* \colon V_Y \to V_X$. By Lemma \ref{tsushimalemma2}, we can recover $f$ from $f^*$ and even from $\mathbb{P}(f^*)$. Thus, we are led to studying linear maps $V_Y \to V_X$.  

Let $H := \Hom(V_Y, V_X)^\vee$, with $^\vee$ denoting the dual. Composition of functions is a canonical bilinear map $V_X^\vee \times \Hom(V_Y, V_X) \to V_Y^\vee$ and after identifying $\Hom(V_Y,V_X)$ with its double dual, we get a bilinear morphism $V_X^\vee \times H^\vee \to V_Y^\vee$. It induces a relative (over $\mathbb{P}(H)$) rational map $F \colon \mathbb{P}(V_X) \times \mathbb{P}(H) \ratmap \mathbb{P}(V_Y)$. For every closed point $h \in \mathbb{P}(H)$, we denote by $F_h$ the rational map $\mathbb{P}(V_X) = \mathbb{P}(V_X) \times \{h\} \ratmap \mathbb{P}(V_Y)$. 

To prove the proposition, we are first going to construct a ``small'' locally closed subset $H_3$ of $\mathbb{P}(H)$ such that the set of separably dominant near-maps (still) injects into $H_3$ via $f\mapsto \mathbb{P}(f^\ast)$.

Let $H_1 \subseteq \mathbb{P}(H)$ be the subset for which $F_h$ maps $\overline{X} \subseteq \mathbb{P}(V_X)$ to $Y \subseteq \mathbb{P}(V_Y)$. To see that this is a meaningful condition, note that the indeterminacy locus of $F_h$ is a linear subspace and that $\overline{X} \subseteq \mathbb{P}(V_X)$ is contained in no proper linear subspace. Let $\eta_{\overline{X}}$ be the generic point of the scheme $\overline{X}$. Since $H_1$ is the set of $h$ in $\mathbb{P}(H)$ such that the morphism
\[
\{\eta_{\overline{X}}\}\times \mathbb{P}(H)\to \mathbb{P}(V_Y)
\]
factors over $Y$, it follows from Lemma \ref{closed_hom} that it is closed in $\mathbb{P}(H)$. 

Note that we obtain a relative rational map $X \times H_1 \ratmap Y$. 
Let $H_2\subset H_1$ be the subset of elements of $H_1$ for which the induced rational map $X \ratmap Y$ is separably dominant. By Lemma \ref{rat_dominant_open}, the set $H_2$ is open in $H_1$.  

Let $H_3\subset H_2$ be the subset of rational maps $g \colon X \ratmap Y$ such that, for every global section $\omega$ of $\omega_Y(\Delta_Y)^{\otimes N}$, the pullback $(g\circ\iota_X)^*\omega$ is a global section of $\omega_X(\Delta_X)^{\otimes N}$. By applying Lemma \ref{prescribed_poles_closed} to every single $\omega$ and taking the intersection over all closed sets obtained this way, we see that $H_3$ is closed in $H_2$. Hence $H_3$ is locally closed in $\mathbb{P}(H)$, and we give it the reduced scheme structure.  

If $f \colon (X, \Delta_X) \ratmap (Y, \Delta_Y)$ is a separably dominant orbifold near-map, then the induced map $\mathbb{P}(f^*)$ lies in $H_3$. As we mentioned before, by Lemma \ref{tsushimalemma2}, different separably dominant orbifold near-maps induce different elements of $H_3$. Therefore, to prove the proposition, it suffices to show that $H_3$ is finite. To do so, let $H_4$ be an irreducible component of $H_3$, so that $H_4$ is a quasi-projective variety. Since $H_3$ is quasi-projective, it has only finitely many irreducible components. Therefore, to conclude the proof, it suffices to show that $H_4$ is finite.

Let $\overline{H_4}$ be the closure of $H_4$ in $\mathbb{P}(H)$, and note that $\overline{H_4}$ is a projective variety. Since $H_1$ is closed in $\mathbb{P}(H)$, we see that $\overline{H_4}$ is contained in $H_1$. In particular, we can interpret every closed point of $\overline{H_4}$ as a (possibly non-dominant) rational map $X \ratmap Y$. Let $\widetilde{H_4}$ be a smooth projective variety and let $\widetilde{H_4} \to \overline{H_4}$ be an alteration such that the preimage of $\overline{H_4} \setminus H_4$ in $\widetilde{H_4}$ is a strict normal crossings divisor $D_H$ (this exists by \cite[Theorem 4.1]{deJongAlterations}).  
We let $G \colon X \times \widetilde{H_4} \ratmap Y$ be the relative rational map induced by the above map $X \times H_1 \ratmap Y$.

By Lemma \ref{direct_summand}, the sheaf $S^N \Omega^n_{(X, \Delta_X) \times (\widetilde{H_4}, D_H)}$ has $\pi_X^* S^N \Omega^n_{(X,\Delta_X)}$ as a direct summand. We denote by 
\[\pi \colon S^N \Omega^n_{(X, \Delta_X) \times (\widetilde{H_4}, D_H)} \to \pi_X^* S^N \Omega^n_{(X,\Delta_X)}\]
the projection. We now define the morphism $\Psi \colon\widetilde{H_4} \to \Hom(V_Y,V_X) = H^\vee$ by $\Psi(h) = \left[ \omega\mapsto \iota_h^*\pi(G^*\omega) \right]$, where $\iota_h\colon X\to X\times \widetilde{H_4}$ is the inclusion map $x\mapsto (x,h)$. We now show that $\Psi$ is a well-defined morphism of varieties.

To do so, fix a closed point $h \in \widetilde{H_4}$ and $\omega \in V_Y$. Then $\omega \in \Gamma(Y, \omega_Y(\Delta_Y)^{\otimes N})$.    
If $h$ does not lie over $H_4$, then the form $G^*\omega$ might have a pole along $X \times \{h\}$, so that the pullback of $G^\ast \omega$ to $X$ is not well-defined. But we do have some control over the poles of $G^*\omega$. Indeed, it can only have poles along $\ceil{\Delta_X} \times \widetilde{H_4}$ with orders bounded by the coefficients of $N \Delta_X$ or poles along $X \times D_H$. The latter poles are logarithmic. This means that $G^*\omega$ is a global section of $S^N \Omega^n_{(X, \Delta_X) \times (\widetilde{H_4}, D_H)}$. In particular, $\pi(G^*\omega)$ is a global section of $\pi_X^* S^N \Omega^n_{(X,\Delta_X)}$. Since global sections of $\pi_X^* S^N \Omega^n_{(X,\Delta_X)}$ only have poles along subsets of $\ceil{\Delta_X} \times \widetilde{H_4}$, the element $\iota_h^*\pi(G^*\omega)$ is always well-defined, i.e., $\Psi \colon \widetilde{H_4} \to H^\vee$ is well-defined. 

The restriction of $\Psi$ to elements of $\widetilde{H_4}$ lying over $H_4$ is simpler to describe. Indeed, if $h$ lies over $H_4$, then we can restrict $G^\ast\omega$ to $X \times \{h\}$. By definition of $H_3$, after identifying $X \times \{h\}$ with $X$, this will give us an element of $V_X = \Gamma(X, \omega_X(\Delta_X)^{\otimes N})$. By Lemma \ref{factor_over_projection}, this element coincides with $\iota_h^*G^*\omega$.

Let $h_1$ and $h_2$ be elements of $\widetilde{H_4}$ lying over $H_4$ such that $\Psi(h_1) = \Psi(h_2)$. Since $\Psi(h_1) \colon V_Y \to V_X$ is the injective map $\omega\mapsto (G\circ \iota_{h_1})^\ast \omega$ and $\Psi(h_2) \colon V_Y\to V_X$ is the injective map $\omega\mapsto (G\circ \iota_{h_2})^\ast \omega$, it follows from Lemma \ref{tsushimalemma2} that the dominant near-maps $G\circ \iota_{h_1}$ and $G\circ \iota_{h_2}$ are equal. This obviously implies that $h_1$ and $h_2$ lie over the same element of $H_4$ (via the alteration $\widetilde{H_4}\to \overline{H_4}$).
 
On the other hand, since $\widetilde{H_4}$ is a projective variety and $H^\vee$ is affine, the morphism $\Psi$ is constant. Since $\Psi$ separates elements lying over distinct points of $H_4$ (see previous paragraph), we conclude that $H_4$ is a singleton. This concludes the proof.
\end{proof}

As we show now, we may drop the properness and smoothness assumptions on $X$.

\begin{corollary}\label{cor0}
Let $(X, \Delta_X)$ and $(Y, \Delta_Y)$ be normal orbifolds with $(Y, \Delta_Y)$ a smooth proper orbifold of general type. If $\dim X = \dim Y$, then there are only finitely many separably dominant near-maps $(X, \Delta_X) \ratmap (Y, \Delta_Y)$. 
\end{corollary}
\begin{proof}
Let $U \subseteq X \setminus \supp \Delta_X$ be a smooth open subset. Let $V \to U$ be an alteration with a smooth compactification $V \subseteq \overline{V}$ (see \cite[Theorem 4.1]{deJongAlterations}). Shrinking $U$ if necessary, we may assume that $V \to U$ is quasi-finite and dominant. Then, by Lemma \ref{lemma:near-map-composition}, the set of separably dominant near-maps $(X, \Delta_X) \ratmap (Y, \Delta_Y)$ injects into the set of separably dominant near-maps $V \ratmap (Y, \Delta_Y)$. Shrinking $V$ if necessary, we may assume that the boundary $\overline{V} \setminus V$ is a (reduced) divisor $D \subset \overline{V}$. Then, the set of separably dominant near-maps $V \ratmap (Y,\Delta_Y)$ equals the set of separably dominant near-maps $(\overline{V}, D) \ratmap (Y,\Delta_Y)$. The latter is finite by Proposition \ref{prop0}.
\end{proof}

 Corollary \ref{cor0} implies that the statement of Theorem \ref{thm1} holds, under the additional assumption that $\dim X = \dim Y$. To prove Theorem \ref{thm1}, we use a cutting argument.

\begin{theorem}\label{thm:final}
Let $(X, \Delta_X)$ and $(Y, \Delta_Y)$ be normal orbifolds with $(Y, \Delta_Y)$ smooth, proper, and of general type. Then there are only finitely many separably dominant near-maps $(X, \Delta_X) \ratmap (Y, \Delta_Y)$. 
\end{theorem}
\begin{proof}
We may assume that the base field $k$ is uncountable. As before, we may replace $(X, \Delta_X)$ by a dense open subset of $X \setminus \supp \Delta_X$, so we may assume that $X$ is smooth and quasi-projective and that $\Delta_X$ is empty. 

We argue by contradiction. Assume that $(f_i \colon X \ratmap (Y, \Delta_Y))_{i \in \mathbb{N}}$ is an infinite sequence of pairwise distinct separably dominant near-maps. If $\dim(X) > \dim(Y)$, we will construct an $(\dim(X)-1)$-dimensional subvariety $Z \subset X$ which still admits infinitely many pairwise distinct separably dominant orbifold near-maps to $(Y, \Delta_Y)$. By iterating this process, we obtain a contradiction to Corollary \ref{cor0}. As the case $\dim(Y)=0$ is trivial, we may assume $\dim(X) \geq 2$. 

Let $X \subseteq \mathbb{P}^n$ be an immersion and let $Z \subset X$ be a very general hyperplane section. Then $Z$ is a smooth variety. Since $Z$ is not contained in the indeterminacy locus of any $f_i$, all of the $f_i$ induce rational maps $f_i \vert_Z \colon Z \ratmap Y$. Moreover, being very general, $Z$ does not contain any irreducible component of the indeterminacy locus of any $f_i$. This implies that all $f_i \vert_Z \colon Z \ratmap Y$ are defined in codimension $1$. For fixed $i, j$, the locus where $f_i$ and $f_j$ are equal is contained in a proper closed subset of $X$, so that $f_i$ and $f_j$ are distinct after restricting to a generic hyperplane section. Thus, the near-maps $(f_i\vert_Z)_{i \in \mathbb{N}}$ are pairwise distinct. By the implication $(a) \implies (b)$ of Lemma \ref{dominant_tangent0}, for every $i \in \mathbb{N}$, there is a point $x \in X$ such that $df_{i,x} \colon \Tangent_x X \to \Tangent_{f_i(x)} Y$ is surjective. As this is an open condition on $x$, we see that the map on tangent spaces is surjective at a general point of $X$. Moreover, if $V \subset \Tangent_x X$ is a generic subspace of codimension $1$, the composed map $V \to \Tangent_{f_i(x)} Y$ will still be surjective, as $\dim \Tangent_x X = \dim X > \dim Y = \dim \Tangent_{f_i(x)} Y$. Consequently, we see that the restricted maps $f_i \vert_Z \colon Z \ratmap Y$ still have a surjective differential map at some point. Hence, they are separably dominant by the implication $(b) \implies (a)$ of Lemma \ref{dominant_tangent0}. By Corollary \ref{corollary:near-map-restriction}, the separably dominant near-maps $f_i \vert_Z \colon Z \ratmap (Y, \Delta_Y)$ are orbifold near-maps. This concludes the proof.     
\end{proof} 

\begin{remark}
Theorem \ref{thm:final} also holds if we allow $(X, \Delta_X)$ and $(Y, \Delta_Y)$ to be $\mathbb{Q}$-orbifolds (but still requiring $(Y, \Delta_Y)$ to be smooth, proper, and of general type). Indeed, as before, we immediately reduce to the case that $X$ is a smooth variety. Writing $\Delta_Y = \sum (1-\frac{1}{m_i}) D_i$, we can define $\widetilde{\Delta_Y} = \sum (1-\frac{1}{\ceil{m_i}}) D$. Then every morphism $X \to (Y, \Delta_Y)$ is also a morphism $X \to (Y, \widetilde{\Delta_Y})$; hence we are reduced to the case of $\mathbb{Z}$-orbifolds. 
\end{remark}

As a straightforward application of our result, we can show that every surjective endomorphism of a general type orbifold is an automorphism.

\begin{corollary}\label{corollary:end}
If $(X, \Delta)$ is a smooth proper orbifold pair of general type, then every separably dominant morphism $(X, \Delta) \to (X, \Delta)$ is an automorphism of finite order, and the group of automorphisms of $(X,\Delta)$ is finite.
\end{corollary}
\begin{proof}
Let $f \colon (X, \Delta) \to (X, \Delta)$ be a separably dominant orbifold morphism. Let $f^n$ be the $n$-fold composition of $f$. Since every $f^n$ is a separably dominant orbifold morphism, by the finiteness of the set of separably dominant orbifold morphisms $(X,\Delta)\to (X,\Delta)$, there are distinct positive integers $m$ and $n$ with $m>n$ such that $f^m = f^n$. As $f$ is dominant, it follows that $f^{n-m} = \id_X$, so that $f$ is an automorphism of finite order. The finiteness of the automorphism group follows immediately from Proposition \ref{prop0}.
\end{proof}

Let $(X,\Delta)$ be a smooth proper orbifold pair of general type.
We do not know whether every  separably dominant orbifold near-map $f\colon (X,\Delta)\ratmap (X,\Delta)$ is birational and has finite order (in the group of birational self-maps of $X$).  Indeed,   it is not clear to us whether the iterates $f^i = f \circ \ldots \circ f$ of $f$ are orbifold near-maps.

\bibliography{refs_tsushima}{}
\bibliographystyle{alpha}

\end{document}